\documentclass{article}
\usepackage{graphicx} 

\usepackage{amsfonts}
\usepackage{amsmath}
\usepackage[utf8]{inputenc}
\usepackage{xy} \xyoption{all}
\usepackage{tikz}
\usetikzlibrary{shapes,arrows}
\usepackage{ wasysym}
\usepackage{amssymb}
\usepackage{mathrsfs}

\usepackage{multicol}

\newtheorem{theorem}{Theorem}

\newtheorem{corollary}[theorem]{Corollary}

\newtheorem{example}[theorem]{Example}
\newtheorem{examples}[theorem]{Examples}

\newtheorem{lemma}[theorem]{Lemma}
\newtheorem{notation}[theorem]{Notation}

\newtheorem{proposition}[theorem]{Proposition}
\newtheorem{remark}[theorem]{Remark}

\newtheorem{fact}[theorem]{Fact}
\newenvironment{proof}[1][Proof]{\noindent\textbf{#1.} }{\ \rule{0.5em}{0.5em}}

\def\title#1{{\Large\bf  \begin{center} #1 \vspace{0pt} \end{center}  } }
\def\authors#1{{\large\bf \begin{center} #1 \vspace{0pt} \end{center} } }
\def\university#1{{\sl \begin{center} #1 \vspace{0pt} \end{center} } }
\def\inst#1{\unskip $^{#1}$}

\begin{document}

\title{Leavitt Path Algebras with Coefficients  in a Commutative Unital Ring}
\authors{Ayten Ko\c{c}\inst{1} 
\quad 
		Murad \"{O}zayd\i n\inst{2}  
	}
	\smallskip
	
	%
	%

			\university{\inst{1}Gebze Technical  University, Türkiye         }
	\university{\inst{2}University of Oklahoma, USA}

\date{August 2023}

\begin{abstract}
 
  In addition to extending some facts from field coefficients to commutative ring coefficients for Leavitt path algebras with new shorter proofs, we also  prove some results that are new even for field coefficients. In particular, we show that the ideal lattice of a Leavitt path algebra embeds into the ideal lattice of the path algebra of the same digraph, we construct a new basis for a Leavitt path algebra of polynomial growth and 
  give a formula for the Gelfand-Kirillov dimension of a Leavitt path algebra in terms of its digraph. 
\end{abstract}

{\bf Keywords:} Leavitt path algebras, quiver representations, Morita equivalence, Gelfand-Kirillov dimension \\
\medskip

{\bf MSC2020:}
 16S88 

 \medskip
 {\bf Secondary:}     16G20, 
16D90, 16P90
 
\medskip
\section{Introduction}
\medskip

The Leavitt path algebra $L(\Gamma)$ of a di(rected )graph $\Gamma$ was defined (many decades after Leavitt's seminal work  \cite{lea65}, via a detour through functional analysis) by Abrams, Aranda Pino \cite{aa05} and by Ara, Moreno, Pardo \cite{amp07} (independently and essentially simultaneously) as an algebraic analog of a graph $C^*$-algebra. In addition to the algebras $L(1,n)$ of Leavitt \cite{lea65}  these include (sums of) matrix algebras (over fields or Laurent polynomial algebras), algebraic quantum discs and spheres, and many others. The important subclass of Leavitt path algebras of polynomial growth are identified as coming from finite digraphs whose cycles are pairwise disjoint and then studied by Alahmedi, Alsulami, Jain, Zelmanov \cite{aajz12}, \cite{aajz13}.

\medskip
\medskip

Initially Leavitt Path Algebras (LPAs) were defined with  field coefficients, this was extended to a commutative ring with 1 in \cite{t11}, see also \cite{L15}. (We will denote a field by $\mathbb{F}$ and a commutative ring with 1 by ${\bf k}$.) Most of the results in this note are for a Leavitt path algebra of a row-finite digraph with coefficients in ${\bf k}$, a commutative ring with 1, defined below. Many of the facts and lemmas needed exist in the literature only for field coefficients and some only for finite digraphs. We provide new, shorter proofs of these basic facts (and sometimes their generalizations) when the coefficients are in ${\bf k}$.

\medskip
\medskip

Some of the generalizations we prove are fairly routine modifications of existing versions. For instance, in \cite{ko1} we showed that the category of (unital) $L_{\mathbb{F}}(\Gamma)$-modules  is equivalent to a full subcategory of quiver representations satisfying a natural isomorphism condition when $\Gamma$ is a  row-finite digraph.  In Section \ref{mod} below we show that the same result is valid when the coefficients are a commutative ring ${\bf k}$ with 1. Similarly, the explicit Morita equivalence given by an effective combinatorial (reduction) algorithm on a  digraph  $\Gamma$ originally given in \cite{ko2} (to classify all finite dimensional modules of a Leavitt path algebra)  also generalizes from $\mathbb{F}$ to ${\bf k}$ as explained below in Section \ref{RA}.

\medskip
\medskip

Section \ref{RA} starts with the proof of the fact that restriction gives an embedding of the ideal lattice of the Leavitt path algebra $L_{\bf k}(\Gamma)$ into the ideal lattice of the path algebra $\bf k\Gamma$, a new result for field coefficients also. While this fact should be of independent interest, here we use it to show that certain homomorphisms from $L_{\bf k}(\Gamma)$ are one-to-one. Previously a similar technique, namely the Graded Uniqueness Theorem \cite[Theorem 2.2.15]{aam17} for field coefficients, \cite[Theorem 5.3]{t11} for commutative ring coefficients was utilized, which could only be used for graded homomorphisms. The rest of Section \ref{RA} is about extending the Reduction Algorithm of \cite{ko2}, giving a Morita equivalence between $L_{\bf k}(\Gamma)$ and a generalized corner subalgebra (which is isomorphic to a Leavitt path algebra of another digraph with less vertices than $\Gamma$).

\medskip
\medskip

In Section \ref{GK} we focus on a finite digraph $\Gamma$ whose cycles are pairwise disjoint (a necessary and sufficient condition for $L_{\mathbb{F}}(\Gamma)$ to have polynomial growth \cite{aajz12}). We define a convenient ${\bf k}$-basis for $L_{\bf k}(\Gamma)$ in Theorem \ref{basis}, different from those in \cite{aajz12} and \cite{ag12}, in particular its proof is independent of any version of the Dimand Lemma. Our basis is useful for giving a combinatorial formula of the Gelfand-Kirillov dimension (Theorem \ref{height}, Theorem \ref{k-height}). This is achieved by defining a height function on the poset of sinks and cycles in $\Gamma$. We compute the Gelfand-Kirillov dimension of $L_{\mathbb{F}}(\Gamma)$ in terms of this height function. Expressing the exact Gelfand-Kirillov dimension purely combinatorially is a new result for field coefficients as well. We also give two generalizations, first when ${\bf k}$ is an $\mathbb{F}$-algebra, second when ${\bf k}$ is an integral domain (using the definition given in \cite{Bell}).

\section{Preliminaries}

\subsection{Digraphs}
\medskip 
A {\em di}({\em rected} ){\em graph} $\Gamma$ is a four-tuple $(V,E,s,t)$ where $V$ is the set of vertices, $E$ is the set of arrows (directed edges), $s$ and $t:E \longrightarrow V$ are the source and  the target functions.

\begin{remark}
A digraph is also called an ”oriented graph” in graph theory, a
”diagram” in topology and category theory, a ”quiver” in representation theory,
usually just a ”graph” in graph $C^*$-algebras and Leavitt path algebras. The notation $\Gamma =(V,E,s,t)$  for a digraph is fairly standard in graph theory, except "$A$" for arrows (or arcs) is also  used instead of "$E$" for edges. In quiver representations $Q = (Q_0, Q_1, s, t)$ is
more common, however "$h$" (denoting head) and "$t$" (denoting tail) and are also used instead of "$t$" and "$s$", respectively. 
In graph $C^*$-algebras $E = (E^0, E^1, s, r)$ is  used. We prefer the graph theory
notation which involves two more letters but no subscripts or superscripts. As
in quiver representations we view $\Gamma$ as a small category, so ”arrow” is preferable
to ”edge”, similarly for ”target” versus ”range” (also range usually means image rather than target).
\end{remark}

The digraph $\Gamma$ is \textit{finite} if $E$ and $V$ are both finite. $\Gamma$ is \textit{row-finite} if $s^{-1}(v)$ is finite for all $v$ in $V$. If $s^{-1}(v)= \emptyset$ then  the vertex $v$ is a \textit{sink}.
If $t(e)=s(e)$ then $e$ is  a \textit{loop}. 
 If $W\subseteq V$ then $\Gamma_{W}$ denotes the \textit{full subgraph} on $W$, that is, $\Gamma_W=(W, s^{-1}(W) \cap t^{-1}(W), s\vert_{_W}, t\vert_{_W})$.

\medskip
\medskip 

A \textit{path} $p$ of length $n>0$ is a sequence of arrows $e_{1}\ldots e_{n}$ such
that $t(e_{i})=s(e_{i+1})$ for $i=1,\ldots ,n-1$. The source of $p$ is $s(p):=s(e_{1})$  and the target of $p$ is $t(p ):=t(e_{n})$.  A path of length 0 consists of a single vertex $v$ where $s(v):=v $ and $t(v) := v$. We will denote the length of $p$ by $l(p)$. An \textit{infinite path} $p$ is an infinite sequence of arrows $e_1e_2 \cdots e_k \cdots $
 with $t(e_i)=s(e_{i+1})$ for each $i$; now $s(p)=s(e_1)$ but $t(p)$ is not defined. A path $C=e_1e_2 \cdots e_n$ with $n>0$ is a \textit{cycle} if $s(C )=t(C )$ and $s(e_{i})\neq s(e_{j})$ for $i\neq j$. We consider the cycles $e_1e_2 \cdots e_n$ and $e_2e_3 \cdots e_ne_1$ equivalent.
 An arrow  $e \in E$ is an \textit{exit} of the cycle $C=e_1e_2 \cdots e_n$ if there is an $i$ such that $ s(e)=s(e_i)$ but $e\neq e_i$. 
 The digraph $\Gamma$ is \textit{acyclic} if it has no cycles. 
  $Path(\Gamma)$ denotes the set of paths in $\Gamma$.

\medskip
\medskip

There is a preorder $\leadsto$ on the set of vertices $V$ of $\Gamma$: $u \leadsto v $ if and only if   there is a path from $ u$ to   $ v$. This preorder defines an equivalence relation on $V$: $u\sim v$ if and only if $u\leadsto v$ and $v\leadsto u$. There is an induced partial order on equivalence classes, which we will also denote by $[u] \leadsto [v]$. When the cycles of $\Gamma$ are pairwise disjoint, the preorder $\leadsto$ defines a partial order on the set of sinks and cycles in $\Gamma$.

\medskip
\medskip

The set of \textbf{predecessors} of $v$ in $V$ is $V_{\leadsto v} := \{ u \in V \mid u \leadsto v \} $. (This set has also been denoted by $M(v)$ in the literature.) In particular, every vertex $v$ is a predecessor of itself since there is a path of length zero from $v$ to $v$. If $u$ and $v$ are two vertices on a cycle $C$ then $V_{\leadsto u}=V_{\leadsto v}$,  so $V_{\leadsto C} :=V_{\leadsto u}$ is well-defined. Let $\Gamma_{\leadsto v}$ and $\Gamma_{\leadsto C}$ be the full subgraphs on $V_{\leadsto v}$ and $V_{\leadsto C}$, respectively. The set of {\bf successors} of $v$ in $V$ is $V_{v \> \leadsto } := \{ w \in V \mid v \leadsto w \} $ (this set has been denoted by $T(v)$ in the literature) and  $V_{C\> \leadsto}$ is defined analogously.
Similarly,  $V_{\leadsto X}:=\{ v \in V \> \vert \> v \leadsto x \textit{ for some } x\in X\}$ and 
$V_{X \> \leadsto} :=\{ v\in V \>\vert \> x \leadsto v \textit{ for some } x \in X\}$
when $X \subseteq V$.

\subsection{Leavitt Path Algebras}

Given a digraph $\Gamma,$ the \textit{extended (or doubled) digraph} of $\Gamma$ is $\tilde{\Gamma} := (V,E \sqcup E^*, s~,t~)$ where $E^* :=\{e^*~|~ e\in E \}$,  the functions $s$ and $t$ are  extended as $s(e^{\ast}):=t(e)$ and $t(e^{\ast }):=s(e) $ for all $e \in E$. Thus the dual arrow $e^*$ has the opposite orientation of $e$. We extend $*$ to an operator defined on all paths of $\tilde{\Gamma}$: Let $v^*:=v$ for all $v$ in $V$, $(e^*)^*:=e $ for all $e$ in $E$ and $p^*:= e_n^* \ldots e_1^*$ for a path $p=e_1 \ldots e_n$ with $e_1, \ldots , e_n $ in $E \sqcup E^* $. In particular $*$ is an involution, that is, $**=id$.

\medskip
\medskip

The \textbf{ Leavitt path algebra} of a digraph $\Gamma$ with coefficients in a commutative ring ${\bf k}$ with 1, as defined in \cite{aa05}, \cite{amp07} and \cite{t11} is the ${\bf k}$-algebra $L_{\bf k}(\Gamma)$ generated by $V \sqcup E \sqcup E^*$ satisfying:\\
\indent(V)  $\quad \quad vw ~=~ \delta_{v,w}v$ for all $v, w \in V ,$ \\
\indent($E$)  $\quad \quad s ( e ) e  = e=e t( e)  $ for all $e \in E\sqcup E^*$, \\
\indent(CK1) $\quad e^*f ~ = ~ \delta_{e,f} ~t(e)$ for all $e,f\in E$, \\
\indent(CK2) $\quad v~  = ~ \sum_{s(e)=v} ee^*$  for all $v$ with $0< \vert s^{-1}(v) \vert < \infty $\\
\noindent
where $\delta_{i,j}$ is the Kronecker delta.
$L_{\bf k}(\Gamma)$ is a $*$-algebra since these relations are compatible with the involution $*$ which defines an anti-automorphism on $L_{\bf k}(\Gamma)$: for all $a$, $b$ in $L_{\bf k}(\Gamma)$ we have $(ab)^*=b^*a^*$. 

\medskip
\medskip
\noindent
{\bf Examples:} \label{example}

$ \bullet_v \qquad \qquad \bullet_u {\stackrel{e}{\longrightarrow}} \bullet_v \quad  \quad \qquad  
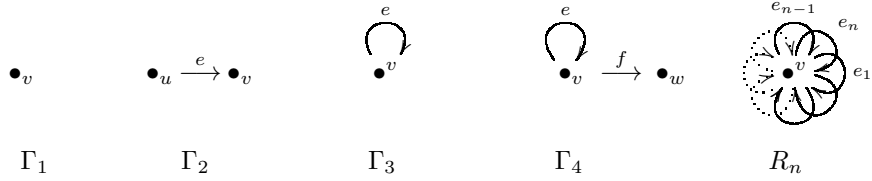
\xymatrix{{\bullet}^{v} \ar@(ul,ur) ^{e}}\quad \quad \xymatrix{  &\> \> { \bullet}_{v} \ar@(ul,ur) ^{e}} \xymatrix{ \stackrel{f}{\longrightarrow} \>   \bullet_{w} }\xymatrix{ & {\bullet^v} 
 \ar@(ur,dr)^{e_1};
 \ar@(u,r)^{e_n};
\ar@(ul,ur)^{e_{n-1}};  
\ar@(dr,dl)^{}; 
\ar@(r,d)^{};  
\ar@{.>}@(dr,dl); 
\ar@{.>}@(d,l);
\ar@{.>}@(dl,ul);
\ar@{.>}@(l,u);
}$ 
$$\Gamma_1 \qquad \quad \quad \quad \Gamma_2 \quad \qquad \qquad  \quad \Gamma_3 \qquad \qquad  \quad \quad \Gamma_4 \qquad\qquad \qquad \quad  R_n  \quad $$

\medskip
\medskip

 (i) $ {\bf k}\Gamma_1={\bf k}v =L_{\bf k}(\Gamma_1)$.

\medskip
\medskip

(ii)  $L_{\bf k}(\Gamma_2) \cong M_2({\bf k})$ where $u \leftrightarrow E_{11}\> , \quad v \leftrightarrow E_{22} \> , \quad e \leftrightarrow E_{12} \> , \quad e^* \leftrightarrow E_{21}$.\\
More generally, if  $\Gamma$ is finite and has no cycles then $L_{\bf k}(\Gamma)$ is isomorphic to a direct sum of matrix algebras $M_n({\bf k})$, each summand corresponds to a sink $w$ and $n=\vert P_w \vert $, the number of paths ending at $w$ as shown in Proposition \ref{aas} below.

\medskip
\medskip

(iii) ${\bf k}\Gamma_3 \cong {\bf k}[x] $ and $L_{\bf k}(\Gamma_3) \cong {\bf k}[x,x^{-1}]$ where $v \leftrightarrow 1 , \quad e \leftrightarrow x^{-1}, \quad e^* \leftrightarrow x$. \\
More generally, if $\Gamma$ is finite and the cycles in $\Gamma$ have no exits then $L_{\bf k}(\Gamma)$ is isomorphic to a direct sum of matrix algebras $M_n({\bf k})$ and/or $M_n({\bf k}[x,x^{-1}])$ with each summand $M_n({\bf k})$ corresponding to a sink as in (ii) and each summand $M_n({\bf k}[x,x^{-1}])$ corresponding to 
a cycle $C$ with $n=\vert P_{C} \vert $ as shown in Proposition \ref{aas} below.

\medskip
\medskip

(iv) $\Gamma_4$ is known as the Toeplitz digraph since by 
completing $L_{\mathbb{C}}(\Gamma_4)$ we can obtain the Toeplitz $C^*$-algebra. $L_{\bf k}(\Gamma_4)$ is isomorphic to the Jacobson algebra ${\bf k}\langle x, y \rangle /(1-xy)$ where: 
$$x \leftrightarrow e^*+f^* \qquad \quad y \leftrightarrow e+f  \qquad \quad 1\leftrightarrow v+w$$ $$v \leftrightarrow yx  \qquad \quad w \leftrightarrow 1-yx  \qquad \quad 
e \leftrightarrow y^2x  \qquad \quad e^*\leftrightarrow yx^2$$
This is a graded isomorphism of $*$-algebras where $\vert x\vert=-1$, $\vert y\vert=1$ and $x=y^*$.

\medskip
\medskip

(v)  $ {\bf k} R_n \cong {\bf k} \langle x_1, x_2,...,x_n \rangle$ and $ L_{\bf k}(R_n) \cong  {\bf k} \langle x_1,...,x_n, y_1, ...,y_n \rangle /I$ where $I$ is generated by $\sum_{i=1}^n x_iy_i- 1$ and $y_ix_j-\delta_{i,j} $ for $1\leq i,j \leq n$. When ${\bf k}$ is a field, $L_{\bf k}(R_n)$  is  the Leavitt algebra $L_{\bf k}(1,n)$.

\medskip
\medskip

We sometimes suppress the subscript ${\bf k}$ in the notation $L_{\bf k}(\Gamma)$. If the digraph $\Gamma$ is fixed and clear from the context we may abbreviate $L(\Gamma)$ to $L$. From now on we will omit the parentheses when the source and target functions $s$, $t$ are applied, to reduce notational clutter.

\medskip
\medskip
The relations (V) simply state that the vertices are mutually orthogonal idempotents. The relations (E) implies that $e \in seLte$ and $e^* \in teLse$ for every $e \in E$. The Leavitt path algebra $L$ as a vector space is $\bigoplus vLw$ where the sum is over all pairs $(v,w) \in V \times V$ since $V\sqcup E\sqcup E^*$ generates $L$. (This direct sum gives the Pierce decomposition of $L$.) If we only impose the relations (V) and ($E$) then we obtain ${\bf k}\tilde{\Gamma}$, the \textbf{path }(or \textbf{quiver}) \textbf{algebra} of the extended digraph $\tilde{\Gamma}$ : The paths in $\tilde{\Gamma}$ form a basis of the free ${\bf k}$-module ${\bf k}\tilde{\Gamma}$. The multiplicative structure of ${\bf k}\tilde{\Gamma}$ is given by:  the product $pq$ of two paths $p$ and $q$ is their concatenation if $tp=sq$ and 0 otherwise, extended linearly. 
$L_{\bf k}(\Gamma)$ is a quotient of $ {\bf k}  \tilde{\Gamma}$ by the ideal generated by the Cuntz-Krieger relations (CK1), (CK2). The algebras $ {\bf k}\tilde{\Gamma}$ and $L_{\bf k}(\Gamma)$ are unital if $V$ is finite, in which case the sum of all the vertices is 1.

\begin{fact} \label{fact}
(i) If $p$ and $q$ are paths in $\Gamma$ then $p^*q=0$ unless 
   $q$ is an initial segment of $p$  (i.e., $p=qr$) or $p$ is an initial segment of $q$. \\
\noindent
(ii) $L_{\bf k}(\Gamma)$ is spanned by $\left\{pq^* \vert \> p,\> q \in Path (\Gamma), \> tp=tq \right\}$ as a ${\bf k}$-module.

\medskip
\medskip

\noindent
(iii) If $C$ is a cycle with no exit then $CC^*=sC=C^*C$. 
\end{fact}

\begin{proof}
(i) If $p$ is not an initial segment of $q$ and $q$ is not an initial segment of $p$ then using the relations (CK1), (E) and (V) when necessary, we can simplify $p^*q$ until we get a term $e^*f$ with $e\neq f$, which is 0 by (CK1).

\medskip
\medskip
(ii) If $p=qr$ then $p^*q$ simplifies to $r^*$, similarly if $q=pr$ then $p^*q=r$. Hence (i) implies that multiplying terms of the form $pq^*$ we get 0 or another term of this form. Also, if $tp\neq tq$ then $pq^*=0$ by (V) and (E). Thus $L_{\bf k}(\Gamma)$ is generated as a ${\bf k}$-module by $pq^*$ with $p, \> q$ in $Path(\Gamma)$ and $tp=tq$.

\medskip
\medskip

(iii) As in (ii), $C^*C=tC=sC$ after applying (CK1), (V) and (E) repeatedly. Since $C$ has no exit,  $se=ee^*$ for each arrow $e$ on $C$ by (CK2). Thus $CC^*=sC$ using (CK2) and (E).
\end{proof}

\medskip
\medskip 
 $L_{\bf k}(\Gamma)$ is a $\mathbb{Z}$-graded $*$-algebra: the $\mathbb{Z}$-grading on the generators is given by $\lvert v \rvert =0$ for $v$ in $V$, $\lvert 
 e \lvert =1$ and $\lvert e^* \rvert =-1$ for $e$ in $E$. This defines a grading on $L_{\bf k}(\Gamma)$ since all the relations are homogeneous. The linear extension of $*$ on paths  to $L_{\bf k}(\Gamma)$ induces a grade-reversing involutive anti-automorphism, that is, $\lvert a^*\rvert =- \lvert a \rvert$ for all homogeneous $a$ in $L_{\bf k}(\Gamma)$.
  The categories of left $L_{\bf k}(\Gamma)$-modules and right $L_{\bf k}(\Gamma)$-modules are equivalent via the anti-automorphism $*$.

\medskip
\medskip

 We may consider $G$-gradings on $L_{\bf k}(\Gamma)$ for any group $G$, with the generators $V \sqcup E \sqcup E^*$ being homogeneous. Since $v^2=v$ and $e^*e=te$ we have: (i) $ \lvert  v \rvert_{_{G}} =1_{_{G}}$ and (ii) $\lvert e^*\rvert_{_G}= \lvert e\rvert^{-1}_{_G}$. Conversely, any function from $V \sqcup E \sqcup E^*$ to $G$ satisfying (i) and (ii) defines a $G$-grading on $L_{\bf k}(\Gamma)$ as the remaining relations are homogeneous. A morphism (or a refinement) from a $G$-grading to an $H$-grading on the algebra $A$ is given by a group homomorphism $\phi:G \longrightarrow H$ such that for all $h \in H$, $A_h= \oplus_{\phi (g)=h} A_g$ where $A_g:= \{a \in A  :\lvert a\rvert_{_{G}}=g \} \cup \{0\} $. There is a universal (or initial) $F_E$-grading on $L_{\bf k}(\Gamma)$ where  $F_E$ is the free group on the set $E$ of arrows, which is a refinement of all others:

\begin{proposition}
Let $F_{E}$ be the free group on $E$. The $F_E$-grading defined on $L_{\bf k}(\Gamma)$ by $\vert v\vert =1$, $\vert e\vert= e$ and $\vert e^*\vert =e^{-1}$ is an initial object in the category of $G$-gradings of 
$L_{\bf k} (\Gamma)$ with the generators $V \sqcup E \sqcup E^*$ being homogeneous.
\end{proposition}

\begin{proof}
For any $G$-grading let $\phi: F_E \rightarrow G$ be the homomorphism given by $\phi(e)= \vert e\vert_{_G}$.
\end{proof}

\medskip
\medskip

Universal grading is used in the proof of Proposition \ref{Path}(i) stating that the path algebra ${\bf k}\Gamma$ may be identified with a subalgebra of $L_{\bf k}(\Gamma)$. When ${\bf k}$ is a field, a (different) proof of this basic fact was originally given in \cite[Lemma 1.6]{goo09}.

\medskip
\medskip

A subset $H$ of $V$ is \textit{hereditary} if $se \in H$ implies that $te \in H$ for all $e \in E$. 
The hereditary closure of $X\subseteq V$ is the set 
$V_{X \> \leadsto}$ of successors of $X$. The subset $H$ is \textit{saturated} if  $te \in H $ for all $e$ in $s^{-1}(v)$ implies that $v\in H$ for each non-sink $v$ with $s^{-1}(v)$ finite. An intersection of hereditary saturated subsets of $V$ is hereditary saturated. The \textit{hereditary saturated closure} of a subset $X$ of $V$, denoted by $\Bar{X}$, is the smallest hereditary saturated subset of $V$ containing $X$, that is, the intersection of all hereditary saturated subsets of $V$ containing $X$. 

\begin{fact} \label{H&S}
If $\Gamma$ is a row-finite digraph then the hereditary saturated closure of $X\subseteq V$ consists of all vertices $v\in V$ such that every path starting at $v$ and ending at a sink and every infinite path starting at $v$ meets a successor of some $x \in X$.  
\end{fact}

\begin{lemma} \label{kesisim}
(i) If $H \subseteq V$ is hereditary then the ideal $(H)$ generated by $H$ is spanned by $\{ pvq^*\> \vert \> p,q \in Path(\Gamma) \text{ and } v \in H\> \}$.\\
(ii) If $X, Y \subseteq V$ then $(X\cap Y) \subseteq (X)(Y)$. \\
(iii) If $H_1$ and $H_2$ are hereditary subsets of $V$ then $(H_1)(H_2)= (H_1\cap H_2)$. 
\end{lemma}

\begin{proof}
(i) The ideal $(H)$ is spanned by elements of the form $p_1q_1^*v p_2q_2^*$ where $p_1,\> q_1, \> p_2, \> q_2$ are in $Path(\Gamma)$ with $tp_1=tq_1, \> sq_1=v=sp_2, \>  tp_2=tq_2$ and $ v\in H$, since the elements of $H$ as well as elements of the form 
 $ pq^*v $ or $vpq^*$ with $ p,\> q \in Path(\Gamma)$ and $ v\in H$  can all  be expressed in the form $p_1q_1^*v p_2q_2^*$.

\medskip
\medskip

By Fact \ref{fact}(i), $q_1^*p_2=q_3^*$ if $q_1=p_2q_3$ or $q_1^*p_2=p_3$ if $p_2=q_1p_3$ or $q_1^*p_2=0$. Since $H$ is hereditary, $v=sp_2 \in H$ gives $tq_2= tp_2 \in H$, similarly $v=sq_1$ gives $tp_1= tq_1 \in H$. Hence $p_1q_1^*vp_2q_2^*$ is $0$ or $p_1p_3(tq_2)q_2^*$ or $p_1(tp_1)q_3^*q_2^*$, proving our claim.

\medskip
\medskip

(ii) This follows since vertices are idempotents.

\medskip
\medskip

(iii) $ (H_1 \cap H_2) \subseteq (H_1)(H_2)$ by (ii) above. Conversely, if $p_1v_1q_1^*p_2v_2q_2^* \neq 0 $ with $v_1 \in H_1$ and $v_2 \in H_2$ then $p_1v_1q_1^*p_2v_2q_2^*=p_3 v q_3^*$ with $v=v_1$ or $v=v_2$ again by Fact \ref{fact}(i). Also $v \in H_1\cap H_2$ because $v_1 , \> v_2  \leadsto v$ and $H_1$, $H_2$ are hereditary. Hence $p_1v_1q_1^*p_2v_2q_2^* \in (H_1\cap H_2)$, yielding $(H_1)(H_2) \subseteq (H_1\cap H_2)$. 
\end{proof}

\medskip
\medskip

The full subgraph on a hereditary saturated $H \subseteq V$ is denoted by $\Gamma_H$. We obtain the subgraph $\Gamma/H$ by deleting all the vertices in $H$ and all arrows touching them, so $\Gamma/H$ is $\Gamma_{V\setminus H}$, the full subgraph on $V\setminus H$.

\medskip
\medskip
When ${\bf k}$ is a field, the following result is well-known   \cite[Lemma 2.4.3 and Corollary  2.4.13(i) ]{aam17}.

\begin{proposition} \label{hereditary} 
Let $\Gamma$ be a row-finite digraph and ${\bf k}$ a commutative ring with 1. \\
(i) If $I$ is an ideal of $L_{\bf k}(\Gamma)$  and $\lambda \in {\bf k}$ then $\{v\in V \> \vert \> \lambda v \in I \}$ is a hereditary saturated subset of $V$.\\      
(ii) If $H$ is a hereditary saturated subset of $V$ and $\Gamma/H$ as above then $L_{\bf k}(\Gamma/H) \cong L_{\bf k}(\Gamma)/ I$ as $\mathbb{Z}$-graded $*$-algebras 
where $I:=(H)$ is the ideal generated by $H$. Also, $vL_{\bf k}(\Gamma/H) \cong vL_{\bf k}(\Gamma)/ vI$ for all $v \notin H$.\\
(iii) If $X\subseteq V$ then $H:=(X) \cap V$ is the hereditary saturated closure of $X$ and $(X)=(H)$. 
\end{proposition}
\begin{proof}
(i) If $e \in E$ with $se =v$
and $\lambda v \in I$ then $
\lambda te =\lambda        e^*e=e^*(\lambda v)e \in I$, hence $\{v\in V \> \vert \> \lambda v \in I \}$ is hereditary. If $v$ is not a sink and $\lambda te \in I$ for all $e$ with $se =v$ then $\lambda v=\lambda \sum ee^*=\sum e(\lambda te)e^* \in I$, hence $\{v\in V \> \vert \> \lambda v \in I \}$ is saturated.

\medskip
\medskip
(ii) We define a $*$-homomorphism from $L_{\bf k}(\Gamma/H)$ to $ L_{\bf k}(\Gamma)/ (H)$ by sending all vertices $v$ to $v+(H)$ and all arrows $e$ to $e+(H)$. The relations $(V)$, $(E)$ and $(CK1)$ are automatically satisfied. If $v$ is not a sink in $\Gamma/H$ then
$$v=\sum_{se=v \>, \> 
te \notin H} ee^* \quad +\sum_{sf=v\>, \>
tf \in H} ff^*$$ 

\noindent
by $(CK2)$ in $L_{\bf k}(\Gamma)$. Since the second sum is in $(H)$ we see that $(CK2)$ is preserved and thus we have a homomorphism.

\medskip
\medskip

We define a $*$- homomorphism from $L_{\bf k}(\Gamma)$  to $L_{\bf k}(\Gamma/H)$ by sending $v$ to $v$ if $v \notin H$ and to 0 otherwise, similarly $e$ to $e$ if $e$ is an arrow in $\Gamma/H$ and to 0 otherwise. The relations $(V)$ and $(E)$ are immediate to verify. Most cases of $(CK1)$ are also easy to check. For the case of $e^*e=te $, if $te \notin H$ then $se \notin H$ since $H$ is hereditary, hence $e$ is in $\Gamma/H$ and the relation holds. When $v$ is not a sink in $\Gamma$ then $v$ is not a sink in $\Gamma/H$ because $H$ is saturated. Breaking up the right-hand side of $(CK2)$ in $L_{\bf k}(\Gamma)$ as above shows that this relation is preserved. The ideal $(H)$ is in the kernel of this homomorphism, so we have the induced homomorphism from $L(\Gamma)/(H)$ to $L(\Gamma/H)$.

\medskip
\medskip
The $*$-homomorphisms above are both $\mathbb{Z}$-graded and are inverses of each other. The last assertion follows from restricting the first homomorphism to $vL_{\bf k}(\Gamma/H)$.

\medskip
\medskip

(iii) By (i) above with $\lambda =1$, $(X)\cap V$ is hereditary saturated. If $H'$ is a hereditary saturated subset of $V$ containing $X$ then $(X)\cap V \subseteq (X) \subseteq (H')$ and $L_{\bf k}(\Gamma/H')\cong L_{\bf k}(\Gamma)/ (H')$ by (ii) above. Hence $(X)\cap V \subseteq H'$ since the vertices in $H'$ are exactly those missing from $\Gamma/H'$. Thus $(X)\cap V$ is contained every hereditary saturated subset containing $H'$, so $H=(X)\cap V$ is the hereditary saturated closure of $X$. Since $X \subseteq H$ and $H \subseteq (X)$ we have $(X)=(H)$.
\end{proof}

\section{$L_{\bf k}(\Gamma)$-modules and Quiver Representations} \label{mod}

\textit{All our modules will be right modules unless specified otherwise.}
 
 \begin{fact} \label{Hom}
If $ \alpha^2=\alpha  \in R$ and $M$ is an $R$-module then $\alpha R$ is a (finitely generated) projective $R$-module and $Hom^{R}(\alpha R, M)\cong M \alpha$ via $f \mapsto f(\alpha)$ and $m \mapsto m \> \centerdot$ where $m \> \centerdot$ is left multiplication by $m \in M\alpha = \{ m \in M \> \vert \> m\alpha=m \}$. In particular, $End^R(\alpha R)\cong \alpha R\alpha$. Similarly, if $I$ is an ideal of $R$ then $Hom^R(\alpha R/\alpha I , M) \cong \{m \in M\alpha \vert \> mI=0 \} $ via $f \mapsto f(\alpha +\alpha I)$ and $m \mapsto m \> \centerdot$ as above.
 \end{fact}

The following result identifying the category of unital $L_{\bf k}(\Gamma)$-modules with a full subcategory of quiver representations of $\Gamma$ enables us to construct many $L_{\bf k}(\Gamma)$-modules and to derive some fundamental properties. A right $R$-module $M$ is \textit{unital} if $ M=MR$ where $R$ may not have 1. This definition agrees with the usual definition of a unital module when $R$ has 1. 

\medskip
\medskip 
 Recall that a quiver representation is a functor from the small category (also denoted by $\Gamma$) with objects $V$ and morphisms given by paths in $\Gamma$, to the category of unital {\bf k}-modules where {\bf k} is a commutative ring with 1 (classically {\bf k} is a field and the target category is finite dimensional vector spaces \cite{dw05}). The coefficient ring ${\bf k}$ is suppressed in the statement of the following proposition. 
 
 \begin{proposition} \label{teorem} 
If $\Gamma= (V,E,s,t)$ is a row-finite digraph then $Mod_{L(\Gamma)}$, the category of unital right $L(\Gamma)$-modules, 
 is equivalent to the full subcategory of quiver representations $\rho$ of $\Gamma$ satisfying the following condition $(Iso)$:

  $$\textit{For every nonsink } v\in V, \> \> \> \>  \bigoplus_{se=v} \rho(e): \rho ( v) \longrightarrow \bigoplus\limits_{se=v } \rho(te) \textit{ is an isomorphism.}$$ The right $L(\Gamma)$-module corresponding to the quiver representation $\rho$ is $\displaystyle M=\bigoplus_{v\in V} \rho(v)$ as a ${\bf k}$-module and the actions of generators of  $L(\Gamma)$ are given by the compositions
  $$- v : M=\bigoplus_{u \in V} \rho(u) \stackrel{pr_{\rho(v)}}{\longrightarrow} \rho(v) \hookrightarrow \bigoplus_{u \in V} \rho(u)=M ,$$
  $$-e : M=\bigoplus_{v\in V} \rho(v) \stackrel{pr_{se}}{\longrightarrow }\rho(se) \stackrel{\rho(e)}{\longrightarrow }\rho(te) \hookrightarrow \bigoplus \rho(v)=M ,$$ 
  $$-e^* : M=\bigoplus_{v\in V} \rho(v) \stackrel{pr_{te}}{\longrightarrow }\rho(te) \hookrightarrow   \bigoplus_{sf=se} \rho(tf) \stackrel{(\bigoplus \rho(f))^{-1}}{ \longrightarrow} \rho(se) \hookrightarrow \bigoplus_{v \in V} \rho(v) =M .$$ 
  
\end{proposition}

The proof of Proposition \ref{teorem} is given in \cite[Theorem 3.2]{ko1} for a field $\mathbb{F}$, but works for a commutative ring ${\bf k}$ with 1. As in the statement of the Proposition \ref{teorem}, from now on $Mod_R$ will denote the category of unital right modules of the ring $R$.

\medskip
\medskip

{\bf In the following all digraphs will be row-finite.}

\begin{example} \label{nonzero p} 
We denote by ${\bf k}X$, the free ${\bf k}$-module with basis $X$. Let $\rho(v)={\bf k}\mathbb{Z}$ for all $v \in V$. Pick an isomorphism $$\varphi_v: \rho(v) \longrightarrow \bigoplus_{se=v} \rho(te)$$ for each non-sink $v\in V$. Let $\rho(e)$ be the composition 
$$ \rho(se)\stackrel{\varphi_{se}}{\longrightarrow} \bigoplus_{sf=se} \rho(tf) \stackrel{pr_e}{\longrightarrow} \rho(te)$$ 
and $\rho(e^*)$ be the composition
$$ \rho(te)\stackrel{\varphi_{se}}{\hookrightarrow} \bigoplus_{sf=se} \rho(tf) \stackrel{\varphi_{se}^{-1}}{\longrightarrow} \rho(se) .$$ 

Since condition (Iso) is satisfied by construction, we have an $L_{{\bf k}}(\Gamma)$-module.
\end{example}

We want to show that most modules obtained from unital modules using standard constructions, such as taking quotients, submodules, sums and extensions,  are also unital. We start with the following:

\begin{lemma} \label{unital}
 Let $\Gamma$ be a row-finite digraph. If $M$ is an $L_{\bf k}(\Gamma)$-module then the following are equivalent:\\
 (i) $M$ is unital.\\
 (ii) $V_m:= \{ v \in V \> \vert \> mv\neq 0\}$ is nonempty and finite for all  nonzero $m$ in $M$. \\
 (iii) $\displaystyle M=\sum_{v \in V} Mv= \bigoplus_{v \in V} Mv$.\\
 (iv) The homomorphism $\displaystyle M \bigotimes_{L_{\bf k}(\Gamma)} L_{\bf k}(\Gamma) \longrightarrow M$ sending $m\otimes a$ to $ma$ is an isomorphism.  
 \end{lemma}
 
 \begin{proof}
 (i) $\Leftrightarrow$ (iii): If $M$ is unital then $m=\sum m_ia_i$ for all $m \in M$ where $a_i=\sum \lambda_{ij}p_{ij}q_{ij}^*$ for finitely many $\lambda_{ij} \in {\bf k}$ and $p_{ij},\> q_{ij}$ in $Path (\Gamma)$. Rearranging $\sum m_i a_i$ by grouping terms with the same $sq_{ij}$ we see that $\displaystyle m \in \sum_{v \in V} Mv$. Since $V$ is a set of orthogonal idempotents, $\sum Mv=\oplus Mv$.

\medskip
\medskip 
 Conversely, if $M=\displaystyle \sum_{v \in V} Mv $ then for all $m \in M$ we have $\displaystyle m= \sum_{i=1}^n m_i$ where $m_i \in Mv_i$ for some $v_1, \> v_2, \> \cdots , \> v_n$ by (iii). Since $m_i=m_iv_i$ and $v_i \in L_{\bf k}(\Gamma)$ we get $\displaystyle m=\sum m_iv_i$, so $M= ML_{\bf k}(\Gamma)$, that is, $M$ is unital.

\medskip
\medskip
 
 (ii) $\Leftrightarrow $ (iii): Assuming (ii), if $m \in M$ then $\displaystyle \sum_{v \in V} mv$ is actually a finite sum. Also $(m-\sum mv)u=mu-mu= 0$ by the relations (V) for all $u \in V$. Hence $V_{m-\sum mv}=\emptyset$ and 
 so $m-\sum mv=0$ by (ii).  Thus $m \in \sum Mv$.

\medskip
\medskip

Conversely, if  $\displaystyle M=\sum Mv$ then $\displaystyle m=\sum_{i=1}^n m_i v_i$ as above, hence $V_m \subseteq \{v_1, \> v_2, \>  \cdots , v_n\}$, thus finite. Also, if $mv=0$ for all $v \in V$ then $m v_i = m_i v_i =0 $ for $i=1,2, \cdots,n$ and hence $m=0$. So, if $0\neq m $ then $V_m \neq \emptyset$.

\medskip
\medskip 
 (iv) $\Rightarrow$ (i): The image of $M \otimes L_{\bf k}(\Gamma) \longrightarrow M$ is $ML_{\bf k}(\Gamma)$, so  $M=ML_{\bf k}(\Gamma)$, that is, $M$ is unital when (iv) holds.

\medskip
\medskip
(ii) $\Rightarrow$ (iv): if (ii) holds then (iii) holds as shown above,  so
 $\displaystyle m \mapsto \sum_{v \in V} (mv \otimes v)$ is the inverse of $M \otimes L_{\bf k}(\Gamma) \longrightarrow M$. 
  \end{proof}
  
  \begin{theorem}
      \label{unital2}
  If $\Gamma$ is a row-finite digraph and ${\bf k}$ is a commutative ring with 1 then the full subcategory of unital $L_{\bf k}(\Gamma)$-modules is closed under colimits and it is a Serre subcategory of all $L_{\bf k}(\Gamma)$-modules (not necessarily unital).
  \end{theorem}
  \begin{proof}
  Clearly, the 0 module is unital. If $M$ is unital, i.e., $M=MR$ then any quotient $N$ of $M$ also satisfies $N=NR$. 
  Using Lemma \ref{unital}(ii) we see that unital modules are closed under taking submodules. Being unital is clearly invariant under isomorphisms. If $A \hookrightarrow M \longrightarrow M/A$ is a short exact sequence of $L_{\bf k}(\Gamma)$-modules with $A$ and $M/A$ unital then $V_m=V_{m+A} \cup V_{a}$ with $a=m-\sum mv$ where the sum is over $v \in V_{m+A}$ for all $m \in M$. Since $m-\sum mv $ is in $A$ and for $m\neq 0$ at least one of $V_{m+A}$ or $V_a$ is nonempty, $M$ is unital by Lemma \ref{unital}(ii). This proves that the subcategory of unital $L_{\bf k}(\Gamma)$-modules is a Serre subcategory.

\medskip
\medskip  
  
  If $M=\bigoplus M_i$ with $M_i$ unital for all $i$ then $\displaystyle M_i= \bigoplus_{v \in V} M_iv$ for each $i$ by Lemma \ref{unital}(iii). Also $Mv=\{ m \in M \> \vert \>  mv =m \}=\bigoplus M_iv$. Changing the order of summation we see that $\displaystyle M=\bigoplus_{v \in V} Mv$. Hence $M$ is unital by Lemma \ref{unital}(iii), so an arbitrary direct sum of unital $L_{\bf k}(\Gamma)$-modules is also unital. Since any colimit is a quotient of a direct sum, we are done. 
  \end{proof}

\medskip
\medskip

Since $L_{\bf k}(\Gamma)$ regarded as a free $L_{\bf k}(\Gamma)$-module is unital, projective modules $vL_{\bf k}(\Gamma)$ with $v \in V$ and their direct sums are all unital by Proposition \ref{unital2}, as well as their quotients. However, the category of unital modules is not closed under arbitrary products. For instance, if $\Gamma$ has infinitely many vertices $v_0, v_1, v_2, \cdots $ then $m \in M=L_{\bf k}(\Gamma)^{\mathbb{N}}$ with $m(i)=v_i$ for $i \in \mathbb{N}$ violates Lemma \ref{unital}(ii), hence $M$ is not unital. 
 
\medskip
\medskip
{\bf From now on all modules will be unital and $Mod_{L_{\bf k}(\Gamma)}$ will denote the category of unital right $L_{\bf k}(\Gamma)$-modules. }

\medskip
\medskip
In the dictionary between a unital $L_{\bf k}(\Gamma)$-module $M$ and a quiver representation $\rho$ of $\Gamma$ satisfying condition (Iso), the ${\bf k}$-module $\rho(v)$ corresponds to $Mv=\{ m \in M \> \vert \> mv=m \}$ for each vertex $v$. So $\displaystyle M\cong \bigoplus_{v \in V} Mv$ as a ${\bf k}$-module.

\medskip
\medskip

Immediate consequences of Proposition \ref{teorem} and Example \ref{nonzero p} are:

\begin{proposition} \label{pp^*}

(i) Let $M$ be a unital $L_{\bf k}(\Gamma)$-module. For every path 
$p$ in $Path(\Gamma)$ the ${\bf k}$-module homomorphism $Msp \stackrel{-p}{\longrightarrow} Mtp$ defined by right multiplication with $p$ is onto and the ${\bf k}$-module homomorphism $Mtp \stackrel{- p^*}{\longrightarrow } Msp$ is one-to-one.

\medskip
\medskip

(ii) For any two paths $p$ and $q$ in $Path (\Gamma)$ with $tp=tq$ and $\lambda \in {\bf k}$, 
 $\lambda pq^*=0$ in $L_{\bf k}(\Gamma)$ if and only if $\lambda=0$.

\medskip
\medskip 
 
 (iii) If $w$ is a sink then $wL_{\bf k}(\Gamma)w ={\bf k}w \cong {\bf k}$.
 \end{proposition}
 
\begin{proof}
 (i) By the condition (Iso) in Proposition \ref{teorem}, $Mse \stackrel{- e}{\longrightarrow} Mte $ is onto for all $e \in E$, hence $Msp \stackrel{- p}{\longrightarrow} Mtp $ is also onto. Similarly, $Mte \stackrel{- e^*}{\longrightarrow} Mse $ is one-to-one, so $Mtp \stackrel{- p^*}{\longrightarrow} Msp $ is also one-to-one.

\medskip
\medskip 
 
(ii) Given an $L_{\bf k}(\Gamma)$-module $M$ as in the Example \ref{nonzero p} above right multiplication by $pq^*$ with $tp=tq$ defines a ${\bf k}$-module homomorphism from $Msp$ to $Msq$ whose image is ${\bf k}\mathbb{Z}$, a free ${\bf k}$-module of infinite rank by (i). Hence $\lambda pq^*=0$ if and only if $\lambda=0$.

\medskip
\medskip
(iii) Since $w$ is a sink and $wL_{\bf k}(\Gamma)w$  is spanned by $pq^*$ with $sp=w=sq$, we see that $p=q=w=pq^*$. Hence $wL_{\bf k}(\Gamma)w ={\bf k}w \cong {\bf k}$ by (ii) above.
\end{proof}

 \begin{notation} \label{not}
 When $C$ is a cycle, $P_C:=\{ p 
\in Path(\Gamma) \> \vert \> tp=sC \textit{ and } p\neq qC \}$.
Let $V_p$ denote the set of all vertices on the path $p$. If $H\subseteq V$ is hereditary then $P_H:=\{p \in Path (\Gamma)\> \vert \> V_p \cap  H = \{ tp\} \}$ and $P_w:= P_{\{w\}}$ if $w$ is a sink. Also, $P_w^u :=\{ p\in P_w \> \vert \> sp=u \}$, 
$\> P_C^u:=\{p \in P_C\> \vert \> sp=u \}$ and $P_H^u:=\{ p \in P_H\> \vert \> sp=u \}$.
\end{notation}

When ${\bf k}$ is a field, part (i) of the following proposition is in \cite[Lemma 1.6]{goo09},  part (ii) is in  \cite[Lemma 2.2.7]{aam17}, part (iv) is in \cite[Theorem 2.6.17]{aam17} and part (vi) is in \cite[Lemma 2.7.1]{aam17}.

\begin{proposition} \label{Path}
(i) The homomorphism from ${\bf k}\Gamma$ to $L_{\bf k}(\Gamma)$ sending every vertex and every arrow to itself is one-to-one. Thus we may view ${\bf k}\Gamma$ as a subalgebra of $L_{\bf k}(\Gamma)$.

\medskip
\medskip

(ii) If $C$ is a cycle with no exit then $vL_{{\bf k}}(\Gamma)v \cong {\bf k}[x,x^{-1}]$ where $v=sC$ and $x \leftrightarrow C^*$.

\medskip
\medskip

(iii) If $w\in V$ is a sink then the $L_{\bf k}(\Gamma)$-module $wL_{\bf k}(\Gamma)$ is free as a ${\bf k}$-module with basis 
$\{p^* \> \vert \> p \in P_w \}$. Moreover, $wL_{\bf k}(\Gamma)$ is a simple $L_{\bf k}(\Gamma)$-module if  and only if ${\bf k}$ is a field.

\medskip
\medskip
(iv) If $w\in V$ is a sink then  
$\{pq^* \> \vert \> p,\> q \in P_w  \}$ is a ${\bf k}$-basis  for the ideal $(w)$ generated by $w$ and $(w)$  is isomorphic to the algebra $M_{P_w}({\bf k})$ of matrices with rows and columns indexed by $P_w$, having finitely many nonzero entries (from ${\bf k}$).

\medskip
\medskip

(v) If $C$ is a cycle with no exit then  $sCL_{\bf k}(\Gamma)$ is a $({\bf k}[x,x^{-1}] ,L_{\bf k}(\Gamma))$-bimodule and it is free as a left ${\bf k}[x,x^{-1}]$-module with basis $ \{p^* \> \vert \> p \in P_C \}$.

\medskip
\medskip
(vi) If $C$ is a cycle with no exit then   
$\{pC^nq^* \> \vert \> n \in \mathbb{Z} \textit{ and } p, q \in P_C \}$ is a ${\bf k}$-basis  for the ideal $(sC)$ generated by $sC$ and $(sC)$  is isomorphic as an algebra to the matrix algebra $M_{P_C}({\bf k}[x,x^{-1}])$ where $C^0:=sC$ and $C^{-n}:=(C^*)^n$ for $n>0$.
\end{proposition}

\begin{proof}
(i) This is a graded homomorphism with respect to the universal grading by the free group on $E$ as described above. Hence the kernel is a graded ideal. Every homogeneous element of ${\bf k}\Gamma$ is of the form $\lambda p$ with $\lambda \in{\bf k}$ and $p \in Path (\Gamma)$. Since $\lambda p=0$ in $L_{\bf k}(\Gamma)$ only if $\lambda =0$ by Proposition \ref{pp^*}(ii), the claim follows.

\medskip
\medskip

(ii) Note that $v$ is the multiplicative identity of the corner algebra $vL_{\bf k}(\Gamma)v$. Paths $p$, $q$ in $ \Gamma$ with $sp=v=sq$ can not leaves $C$ since $C$ has no exit. So, if $tp=tq$ and $p$ and $q$ have positive lengths we have $p=p_1e$ and $q=q_1e$ for some arrow $e$ on $C$. Thus $pq^*=p_1q_1^*$ since $se=ee^*$ for all $e$ on $C$ by (CK2). Repeating this we see that $vL_{\bf k}(\Gamma)v$ is spanned by $C^n$ where $C^0:=v$ and $C^{-1}:= C^*$ by Fact \ref{fact}(iii). We define an epimorphism from ${\bf k}[x,x^{-1}]$ to $vL_{\bf k}(\Gamma)v$ sending $x$ to $C^*$.

\medskip
\medskip

This epimorphism is one-to-one because $\{C^n\> \vert \> n \in \mathbb{Z}\}$ is linearly independent: A finite subset of  $\{C^n\> \vert \> n \in \mathbb{Z}\}$ is mapped to a set of distinct paths in $\Gamma$ after multiplying by a sufficiently high power of $C$, which are linearly independent by (i).

\medskip
\medskip

(iii) Since $w$ is a sink, $wL_{{\bf k}}(\Gamma)$ is spanned by  $\{p^* \> \vert \> p \in P_w \}$. Applying the anti-automorphism $*$ to this set we get a ${
\bf k}$-linearly independent set by (i). Hence  $wL_{\bf k}(\Gamma)$ is a free ${\bf k}$-module with basis $P_w^*$.

\medskip
\medskip
If ${\bf k}$ is not a field then $\mathfrak{m}wL_{\bf k}(\Gamma)$ is a non-zero proper submodule of $wL_{\bf k}(\Gamma)$ where $\mathfrak{m}$ is a maximal ideal of ${\bf k}$, hence $wL_{\bf k}(\Gamma)$ is not simple.
If ${\bf k}$ is a field and $0\neq M$ is a submodule of $wL$ then  there is an $m=\sum \lambda_i p_i^* \in M$ with $\lambda_1\neq 0$ and $p_i$ distinct with $tp_i=w$. Now $m(\frac{1}{\lambda_1}p_1)=w $ by Fact  \ref{fact}(i) and $w$ generates $wL_{{\bf k}}(\Gamma)$. Hence $M=wL_{{\bf k}}(\Gamma)$ showing that $wL_{{\bf k}}(\Gamma)$ is simple.

\medskip
\medskip

(iv) Since $\{w\}$ is hereditary, 
$\{pq^*=pwq^* \> \vert \> p,\> q \in P_w  \}$ spans $(w)$ by Lemma \ref{kesisim}(i). As a ${\bf k}$-module $wL_{\bf k}(\Gamma) \cong {\bf k}P_w^*$ by (iii) above. The elements 
$pq^* \in (w)$ with $ p, q \in P_w$ act on $wL_{\bf k}(\Gamma)$ as elementary matrices  $E_{pq} \in M_{P_w}({\bf k})$. Hence  $\{pq^*\> \vert \> p,\> q \in P_w  \}$ is linearly independent over ${\bf k}$ and $(w)\cong M_{P_w}({\bf k})$ as an algebra.

\medskip
\medskip

(v) $sCL_{\bf k}(\Gamma)$ 
is a $(sCL_{\bf k}(\Gamma)sC ,L_{\bf k}(\Gamma))$-bimodule and $sCL_{\bf k}(\Gamma)sC \cong {\bf k}[x,x^{-1}]$ by (ii) above. 
If $pq^* \in L_{\bf k}(\Gamma)$ with $sp=sC$ and $tp=tq$ then $tp \in V_C$ since $C$ has no exit. Such $pq^*$ spans $sCL_{\bf k}(\Gamma)$. By repeated applications of (CK2) as needed we may assume that $tp=sC=tq$ (again since $C$ has no exit). Using $CC^*=sC$ we may now express such $pq^*$ as $C^nr^*$ with $r \in P_C$ and $n \in \mathbb{Z}$. Hence 
$\{C^np^* \> \vert \> n \in \mathbb{Z}\textit{ and } \> p \in P_C \}$ spans $sCL_{\bf k}(\Gamma)$. Thus any element in $sCL_{\bf k}(\Gamma)$ can be expressed as $\sum_{i=1}^m \alpha_i p_i^*$ with $\alpha_i \in sCL_{\bf k}(\Gamma)sC \cong {\bf k}[x,x^{-1}]$ and $p_i$ distinct elements of $P_C$. We can recover $\alpha_j$  as $(\sum_{i=1}^m \alpha_i p_i^*)p_j$ for all $j$
by Fact \ref{fact}(i) 
since distinct $p_i$ can not be initial segment of one another. Therefore $sCL_{\bf k}(\Gamma)$ is a free ${\bf k}[x,x^{-1}]$-module with basis $\{p^* \> \vert \> p \in P_C \}$.

\medskip
\medskip

(vi) Since $(sC)=(V_C)$ and $V_C$ is hereditary (because  
$C$ has no exit), 
$\{p_1q_1^* \> \vert \> p_1,\> q_1 \in P_C ,\> tp_1=tq_1 \in V_C \}$ spans $(sC)$ by Lemma \ref{kesisim}(i). As in the proof of (v) above, such $p_1q_1^*=pq^*$ with $tp=sC=tq$ and  $(sC)$ is spanned over ${\bf k}$ by $\{pC^nq^* \>\vert \> n\in \mathbb{Z} \textit{ and } p,q\in P_C\}$.
 By (v) above,  $sCL_{\bf k}(\Gamma) \cong {\bf k}[x,x^{-1}]P_C$.   
 As a right $(sC)$-module, we see that the elements 
$pC^nq^*$ act on $sCL_{\bf k}(\Gamma)$  as $x^nE_{pq} \in M_{P_C}({\bf k}[x,x^{-1}])$
 via the isomorphism 
 $sCL_{\bf k}(\Gamma)\cong {\bf k}[x,x^{-1}]P_C$ 
  for all $n \in \mathbb{Z}$ and $ p, q \in P_C$.
 Hence  $\{pC^nq^*\> \vert \> n \in \mathbb{Z} \textit{ and } \> p,\> q \in P_C  \}$ is linearly independent over ${\bf k}$ and $(sC)\cong M_{P_C}({\bf k}[x,x^{-1}])$ as an algebra. 
\end{proof}

\medskip
\medskip
The special case of the following proposition when ${\bf k}$ is  a field and $\Gamma$ is finite is given as Theorems 1.8 and 3.8 in \cite{aam08}.

\begin{proposition} \label{aas}
If ${\bf k}$ is a commutative ring with 1 and $\Gamma$ is a row-finite digraph such that every infinite path is of the form $pC^{\infty}$ where $C$ is a cycle with no exit then $$L_{\bf k}(\Gamma)= \left(\bigoplus_{w \> sink} (w)\right)\> \bigoplus \> \left(\bigoplus_{C \> cycle} (sC)\right).$$
Also $(w) \cong M_{P_w}({\bf k})$ and  $(sC) \cong M_{P_C}({\bf k}[x,x^{-1}])$.
    
\end{proposition}

\begin{proof}
From the hypothesis on the infinite paths it follows that the cycles in $\Gamma$ have no exit. We may replace $pq^*=p(tp)q^*$ in $L_{\bf k}(\Gamma)$ with $p (\sum_{se=tp} ee^*)q^*$ if $tp$ is neither a sink nor an $sC$ for some cycle $C$, using (CK2). We repeat this as many times as needed (this is a finite process by our hypothesis on infinite paths). Therefore  
$L_{\bf k}(\Gamma)$ is the sum of the ideals
generated by a sink $w$ or by $sC$ for some cycle $C$ by Fact \ref{fact}(ii). Using the bases in Proposition \ref{Path}(iv) and /or (vi) and Fact \ref{fact}(i) we see that this sum is direct.
\end{proof}

   \medskip
\section{The Reduction Algorithm } \label{RA}
\medskip

The following Proposition is useful for showing that certain epimorphisms are isomorphisms. Below ${\bf k}\Gamma^*$ denotes the path algebra of $\Gamma^*:= (V, E^*)$.

\begin{proposition} \label{bire-bir}
(i) If $I$ and $J$ are right (respectively, left) ideals of $L_{\bf k}(\Gamma)$ then $I=J$ if and only if $I\cap {\bf k}\Gamma= J \cap  {\bf k}\Gamma $ (respectively, $I \cap {\bf k}\Gamma^*= J \cap  {\bf k}\Gamma^*$). In particular, the right (respectively, left) ideal $I=0$ if and only if  $I\cap {\bf k}\Gamma =0$ (respectively, $I\cap {\bf k}\Gamma^* =0$). 

\medskip
\noindent
(ii) A ${\bf k}$-algebra homomorphism $\varphi$ from $L_{\bf k}(\Gamma)$ to a ${\bf k}$-algebra is one-to-one if and only if the restriction of $\varphi$ to ${\bf k}\Gamma$ is one-to-one.  If $\varphi$ is a $\mathbb{Z}$-graded homomorphism then $\varphi$ is one-to-one if and only if the restriction of $\varphi$  to ${\bf k}v$ is one-to-one for all $v \in V$.

\medskip
\noindent
(iii) The map $\underline{\>\> \>} \cap {\bf k}\Gamma$ from the ideal lattice (respectively, the right ideal lattice) of $L_{\bf k}(\Gamma)$ to the ideal lattice (respectively, the right ideal lattice) of ${\bf k}\Gamma$ is a lattice monomorphism. Similarly, $\underline{\>\> \>} \cap {\bf k}\Gamma^*$ from the left ideal lattice of $L_{\bf k}(\Gamma)$ to the left ideal lattice of $\bf k\Gamma^*$ is a lattice monomorphism.
\end{proposition}

\begin{proof}
(i) If $I \neq J$ then we may assume that there is an $\displaystyle \alpha = \sum_{i=1}^n \lambda_i p_iq_i^* \in I\setminus J $. Since $\alpha \sum v =\alpha $  where the sum is over $\{ v=sq_i \> \vert \> 1\leq i \leq n \}$ there is a vertex $w$ with $\alpha w \in I \setminus J$. If $w$ is a sink then $\alpha w \in ({\bf k}\Gamma \cap I)\setminus ({\bf k}\Gamma \cap J)$ and we are done. If $w$ is not a sink then $\displaystyle \alpha \sum_{se=w} ee^*=\alpha w \in I\setminus J$, so there is $e\in E$ with $\alpha e \in I\setminus J$. This shortens the $q_i^*$s and, repeating this we end up with an element in $(I\cap {\bf k}\Gamma ) \setminus (J\cap {\bf k}\Gamma )$. Thus $I\cap {\bf k}\Gamma \neq J\cap {\bf k}\Gamma $. If $I \neq J$ are left ideals of $L_{\bf k}(\Gamma)$ then $I^*\neq J^*$ are right ideals and we can apply the discussion above to get 
$I^*\cap {\bf k}\Gamma  \neq J^*\cap {\bf k}\Gamma $, hence 
$I\cap {\bf k}\Gamma^*  =(I^*\cap {\bf k}\Gamma)^* \neq (J^*\cap {\bf k}\Gamma)^*=J\cap {\bf k}\Gamma^* $. 

\medskip
\medskip

(ii) If $\varphi$ is a one-to-one homomorphism from $L_{\bf k}(\Gamma)$ to a ${\bf k}$-algebra then every restriction of $\varphi$ is also one-to-one. Conversely, if $\varphi$ is not one-to-one then $0\neq Ker \varphi$. Hence $Ker \varphi \cap {\bf k}\Gamma \neq 0$ 
and the restriction of $\varphi$ to ${\bf k}\Gamma$ is not one-to-one. If $\varphi$ is a graded homomorphism then $0\neq Ker \varphi$ is a graded ideal, so we can find $0\neq \sum \lambda_i p_i \in Ker \varphi \cap {\bf k}\Gamma$ with $p_i$ distinct, $\lambda_i \neq 0$ for each $i$ and all paths $p_i$ having the same length. 
Now $p_1^*\sum \lambda_ip_i =\lambda_1 sp_1 \in {\bf k}sp_1 \cap Ker \varphi$ because $p_1^*p_i=0$ for $i=2,3, \cdots , n$
by Fact \ref{fact}(i), since $p_i$s have the same length.

\medskip
\medskip
(iii) Let $I$ and $J$ be (right) ideals of $L_{\bf k} (\Gamma)$. By (i) above,  $\underline{\>\> \>} \cap {\bf k}\Gamma$ is one-to-one. Since  $(I \cap {\bf k}\Gamma)\cap (J \cap {\bf k}\Gamma )= (I\cap J) \cap {\bf k}\Gamma$ meet is preserved. Clearly,  $(I\cap {\bf k}\Gamma) +  (J \cap {\bf k}\Gamma)\subseteq (I+J) \cap {\bf k} \Gamma$. If $\alpha \in (I+J)\cap {\bf k}\Gamma$ then $\alpha =\beta +\gamma$ with $\beta \in I$ and $\gamma \in J$. Assuming to the contrary that $(\beta +\gamma) \notin I\cap {\bf k}\Gamma + J\cap {\bf k} \Gamma$, we can find a path $p$ in $\Gamma$ with $(\beta +\gamma)p \notin I\cap {\bf k}\Gamma +J \cap {\bf k}\Gamma$ with $\beta p$ and $\gamma p$ in ${\bf k}\Gamma$, as in the proof of (i) above. Now, $\beta p \in I\cap {\bf k}\Gamma$ and $\gamma p \in J\cap {\bf k}\Gamma$ giving a contradiction. Therefore $(I\cap {\bf k}\Gamma) +  (J \cap {\bf k}\Gamma) = (I+J) \cap {\bf k} \Gamma$, that is, join is also preserved. Hence we have a lattice monomorphism from the (right) ideal lattice of $L_{\bf k} (\Gamma) $ to the (right) ideal lattice of ${\bf k}\Gamma$. The proof for $\underline {\> \> \>}\cap {\bf k}\Gamma^*$ is very similar. 
\end{proof}

\medskip
\medskip

When ${\bf k}$ is a field, the last claim in Proposition \ref{bire-bir}(ii) above about graded homomorphisms is known as the Graded Uniqueness Theorem \cite[Theorem 2.2.15]{aam17}. The generalization to a unital commutative ring ${\bf k}$ is  given in \cite[Theorem 5.3]{t11}. We have included a short proof here since it follows immediately from the main assertion of Proposition \ref{bire-bir}(ii).

\medskip
\medskip

The restriction map $\underline{\>\>\>}\cap {\bf k}\Gamma$ from the ideals of $L_{\bf k}(\Gamma)$ to the ideals of ${\bf k}\Gamma$ is one-to-one by Proposition \ref{bire-bir}, but not onto in general. For instance, when $\Gamma$ is a single loop, the ideal $(X)$ of $ {\bf k}[x] \cong {\bf k}\Gamma$ is not the restriction of any ideal of $L_{\bf k}(\Gamma) \cong {\bf k}[x,x^{-1}]$. 

\begin{remark} \label{karsit}
    If $R$ is a subring of $S$ then $\underline{\> \> \>} \cap R$ may not preserve "join": If $R = \mathbb{F}[x] \hookrightarrow \mathbb{F}[x,y]=S$ then $(y)\cap \mathbb{F}[x] +(1-y) \cap \mathbb{F}[x] =0$ but $\bigl( (y)+(1-y) \bigr) \cap \mathbb{F}[x]= \mathbb{F}[x]$.
\end{remark}

\begin{corollary} \label{bire-bir2}
When $\Lambda$ is a subgraph of a digraph $\Gamma$ with the property that if $e$ is an arrow in $\Lambda$ and $se=sf$ for an arrow $f$ in $\Gamma$, then $f$ is an arrow in $\Lambda$. Sending each vertex, each arrow and each dual arrow in $L_{\bf k}(\Lambda)$ to itself in $L_{\bf k}(\Gamma)$ defines a one-to-one $*$-homomorphism.
\end{corollary}
\begin{proof}
To see that this defines a homomorphism $\varphi$, all relations defining $L_{\bf k}(\Lambda)$ need to be checked. This is immediate for all but the relations (CK2), which follow from the condition on $\Lambda$. Clearly, $\varphi$ is a $*$-homomorphism. The restriction of $\varphi$ to ${\bf k}\Lambda$  is one-to-one since the set of paths is linearly independent in $L_{\bf k}(\Gamma)$ by Proposition \ref{Path}(i). Hence $\varphi$ is one-to-one by Proposition \ref{bire-bir}.
\end{proof}

\medskip
\medskip
Corollary \ref{bire-bir2} allows us to identify $L_{\bf k}(\Lambda)$ with a subalgebra of $L_{\bf k}(\Gamma)$ when the subgraph $\Lambda$ satisfies the condition that if $e$ is an arrow in $\Lambda$ and $se=sf$ for an arrow $f$ in $\Gamma$, then $f$ is an arrow in $\Lambda$. For instance, if $H$ is a hereditary subset of $V$ then $\Gamma_H$, the full subgraph on $H$, satisfies this condition.

\medskip
\medskip
Now we consider the consequences of a geometric (graph theoretic)  process we call the \textbf{reduction algorithm} (\cite{ko1}, \cite{ko2}) defined on a row-finite digraph $\Gamma=(V,E)$: For a loopless nonsink $v \in V$, we replace each path $fg$ of length 2 such that $tf=v=sg$ with an arrow labeled $fg$ from $sf$ to $tg$ and delete $v$ and all arrows touching $v$. (Note that $fg$ denotes a path in $\Gamma$, but an arrow in its reduction.) In particular, if $v$ is a source but not a sink, then we delete $v$ and all arrows starting at $v$ without adding any new arrows. We may repeat this  as long as there is a loopless non-sink. Any digraph obtained during this process is called a \textbf{reduction} of $\Gamma$. 
If $\Gamma$ is finite, after finitely many steps we will reach a \textbf{complete reduction} of $\Gamma$,  which has no loopless nonsinks. A digraph in which every vertex is either a sink or has a loop, is called \textbf{completely reduced}.

\medskip
\medskip
In the example below, $\Gamma_1$, $\Gamma_2$ and $\Gamma_3$ are reductions of the digraph $\Gamma$. The number of arrows from one vertex to another is indicated by the number above the arrow (so, in $\Gamma$ there are 3 arrows from $v$ to $w$).  $\Gamma_3$ is a complete  reduction of $\Gamma$.

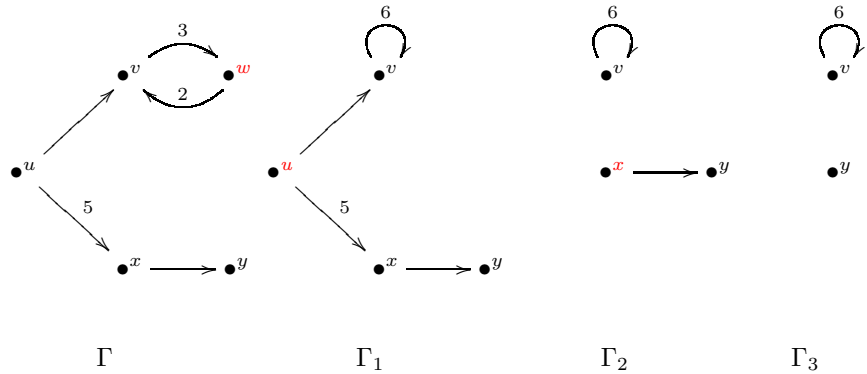
\begin{example}

$$ \xymatrix{&   {\bullet}^{v} \ar@/^1pc/[r]^{3} & {\bullet}^{\text{\textcolor{red}{w}}} \ar@/^1pc/[l]_{2} \\ {\bullet}^{u} \ar@{->}[ur] \ar@{->}[dr]^{5} & \\
                & \bullet^{x} \ar@{->}[r]  & \bullet^{y}  }
 \xymatrix{&{\bullet}^{v}\ar@(ul,ur)^{6} & \\ {\bullet}^{\text{\textcolor{red}{u}}} \ar@{->}[ur] \ar@{->}[dr]^{5}& \\
               & \bullet^{x} \ar@{->}[r]  & \bullet^{y}  }
 \xymatrix{&{\bullet}^{v}\ar@(ul,ur)^{6} & \\ 
                & \bullet^{\textcolor{red}{x} } \ar@{->}[r]  & \bullet^{y} }
 \xymatrix{ &{\bullet}^{v} \ar@(ul,ur)^{6}\\ 
            &  \bullet^{y} }$$

 $$\qquad  \Gamma \qquad\qquad  \qquad  \quad  \qquad  \Gamma_1\qquad  \qquad \qquad\qquad  \Gamma_2\qquad  \qquad \qquad \Gamma_3 $$
\end{example}

The complete reduction $\Gamma_3$ of the digraph $\Gamma$ above does not depend on the choice of reductions. However, complete reductions are not unique (up to digraph isomorphism) in general. For instance, the digraph $\Lambda$ below has two non-isomorphic complete reductions:

$$\Lambda : \  \quad  
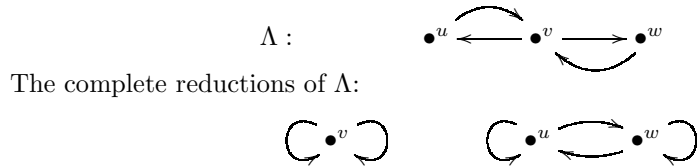
\xymatrix{& {\bullet}^{u} \ar@/^1pc/[r] &\ar [l]  {\bullet}^{v}  \ar [r]  & {\bullet}^{w} \ar@/^1pc/[l] }$$

\noindent
The complete reductions of $\Lambda$:
$$   \xymatrix{ & {\bullet}^{v} \ar@(ul,dl) \ar@(ur,dr)  } \qquad \quad 
 \xymatrix{& {\bullet}^{u}\ar@(ul,dl) \ar@/^0.6pc/[r] & {\bullet}^{w} \ar@(ur,dr)    \ar@/^0.5pc/[l] }\quad $$

A digraph and all reductions of it have the same set of sinks. Cycles may get shorter under each reduction and but they do not disappear. If $\Gamma$ is finite and the cycles in $\Gamma$ are pairwise disjoint then $\Gamma$ has a unique complete reduction up to isomorphism. All cycles in $\Gamma$ become loops in the complete reduction. The vertices of the complete reduction  correspond to the sinks and the cycles in $\Gamma$. 

\begin{theorem}  \label{Reduced}
Let ${\bf k}$ be a commutative ring with 1. If  $\Lambda$ is a  reduction of $\Gamma$, then $L_{\bf k}(\Gamma)$ and $L_{\bf k}(\Lambda)$ are Morita equivalent, that is, their unital module categories are equivalent. $L_{\bf k}(\Lambda)$ is isomorphic to a subalgebra of $L_{\bf k}(\Gamma)$ and a Morita equivalence is given  by the restriction functor from $Mod_{L_{\bf k}(\Gamma)}$ to $Mod_{L_{\bf k}(\Lambda)}$. \end{theorem}

  \begin{proof}
      The proof of the first claim  when ${\bf k}$ is a field is given in 
 \cite[Theorem 4.1]{ko2}. This proof also works over a commutative ring with 1, similar to the proof of Proposition \ref{teorem} (which is the main tool of this proof). 

\medskip
\medskip
 
It suffices to prove that $L_{\bf k}(\Lambda)$ is isomorphic to a subalgebra of $L_{\bf k}(\Gamma)$ when $\Lambda$ is obtained from $\Gamma$ by eliminating a single (loopless nonsink) $v$ in  $V$. We define a $*$-algebra homomorphism $\varphi$ from $L_{\bf k}(\Lambda)$ to $L_{\bf k}(\Gamma)$ by sending each vertex and each original arrow to itself and a new arrow $ef$ of $\Lambda$ to the path $ef$ in $\Gamma$. It's routine to check that the relations are satisfied, checking (CK2)  uses (CK2) of $\Gamma$ twice if new arrows are involved.

\medskip
\medskip
The image of $\varphi$ is the subalgebra of $L_{\bf k}(\Gamma)$ spanned by 
   $\{pq^* \> \vert \>  sp \neq v \neq sq \}$ as ${\bf k}$-module: This set is closed under multiplication by Fact \ref{fact}. If $pq^*\in L_{\bf k}(\Gamma)$ with $sp \neq v\neq sq$, but $tp=v=tq$ then $pq^*=pvq^*=p(\sum_{sf=v} ff^*)q^*$ since $v$ is not a sink. Note that $p$ and $q$ are paths of positive length because $sp\neq v =tp$ and $sq\neq v =tq$, hence $p=p'e$ and $q=q'g$ with $ef$ and $gf$ being images of new arrows in $\Lambda$ for all $f$ with $sf=v$. Similarly, if $p'$ or $q'$ pass through $v$ then we have a path of length 2 corresponding to a new arrow in $\Lambda$, using $tp'=se\neq v \neq sg=tq'$ since $v$ is loopless.

\medskip
\medskip
   
   The restriction of $\varphi$ to ${\bf k}\Lambda$ is one-to-one, so $\varphi$ is one-to-one by Propositon \ref{bire-bir}. Hence $\varphi$ defines an isomorphism to 
   the subalgebra of $L_{\bf k}(\Gamma)$ spanned by 
   $\{pq^* \> \vert \>  sp \neq v \neq sq \}$.

\medskip
\medskip

 From the quiver representation viewpoint, reduction corresponds to restricting the representation to the remaining vertices. Hence the restriction functor from $Mod_{L_{\bf k}(\Gamma)}$ to $Mod_{L_{\bf k}(\Lambda)}$ gives a Morita equivalence via Proposition \ref{teorem}. 
 The new arrow $ef$ is assigned the composition of the ${\bf k}$-module homomorphisms assigned to $e$ and $f$.  We recover the original representation $\rho$ by assigning 
$\displaystyle{\bigoplus_{se=v}} \rho(te)$ to the deleted vertex $v$. The details are given in \cite[Theorem 4.1]{ko2}.
 \end{proof}

\medskip
\medskip
       
   When $\Gamma$ is finite, a reduction $\Lambda$ of $\Gamma$ is isomorphic to the corner algebra $(\sum u)L_{\bf k}(\Gamma) (\sum u) $ where the sum is over the vertices of $\Lambda$, i.e., all the vertices that were not eliminated during the reduction process. Actually the expression $(\sum u)L_{\bf k}(\Gamma) (\sum u) $ makes sense even if $\Gamma$ is infinite because the product of all but finitely many vertices with an element of $L_{\bf k}(\Gamma)$ is $0$. 
   
   \begin{corollary} \label{corner1}
   If $\Lambda$ is a reduction of $\Gamma$ then 
   $L_{\bf k}(\Lambda)\cong (\sum u)L_{\bf k}(\Gamma) (\sum u)$ where the sum is over the vertices of $\Lambda$ as above.
\end{corollary}   \begin{proof}
The subalgebra of $L_{\bf k}(\Gamma)$ spanned by 
   $\{pq^* \> \vert \>  sp ,\>sq \in V_{\Lambda}\}$
  as in the proof of Theorem \ref{Reduced} above is $(\sum u)L_{\bf k}(\Gamma) (\sum u)$.
\end{proof}
       
  \begin{corollary}
 Let $\Gamma$ be a finite digraph with pairwise disjoint cycles and let $\Lambda$ be the complete reduction of $\Gamma$. If $U$ is a set containing all the sinks in $\Gamma$ and exactly one vertex from each cycle in $\Gamma$ then $$L_{\bf k}(\Lambda) \> \cong \>  \big(\sum_{u \in U} u\big)L_{\bf k}(\Gamma) \big(\sum_{u \in U} u\big) .$$
  \end{corollary}     
       
   \begin{proof}
   A complete reduction of a finite digraph whose cycles are pairwise disjoint is gotten by eliminating all the vertices other than the sinks and one (arbitrarily  chosen) vertex from each cycle. Applying Corollary \ref{corner1} yields the result.
   \end{proof}

\section{Gelfand-Kirillov Dimension of $L_{\bf k}(\Gamma)$} \label{GK}

In this section we start by constructing a ${\bf k}$-basis for a Leavitt path algebra $L_{\bf k}(\Gamma)$ of a finite digraph $\Gamma$ whose cycles are pairwise disjoint.
Even when {\bf k} is a field, this basis is different from the Gr\"{o}bner-Shirshov basis in \cite{aajz12} or its generalization in \cite{ag12}. In particular, our proof does not use any version of the Demand Lemma.

\medskip
\medskip 

Recall that when the cycles in $\Gamma$ are pairwise disjoint, the preorder $\leadsto$ defines a partial order on the set of sinks and cycles in $\Gamma$.

\begin{theorem} \label{basis}
If $\Gamma$ is a finite digraph whose cycles are pairwise disjoint then $L_{\bf k}(\Gamma)$ is free as a ${\bf k}$-module with basis 
$$\mathcal{B} :=\{ pq^*\> \vert \> tp=tq \textit{ is a sink } \} \cup \{ pC^nq^* \> \vert \> p, q \in P_{C} ,\> n \in \mathbb{Z},\>  C \textit{ is a cycle   } \}$$ 
where $C^0:=sC$  and $C^{-n}:=(C^*)^n$ for $n>0$.
\end{theorem}

\begin{proof}
We will use induction on the total number $m$ of sinks and cycles in $\Gamma$. When $\Gamma$ has a unique sink $w$ and no cycles,  each  $p_1q_1^* \in L_{\bf k}(\Gamma)$ with $tp_1=tq_1$ can be expressed as a linear combination of $pq^*$'s with $tp=w=tq$ after applying (CK2) to $tp_1$ as needed. Hence $L_{\bf k}(\Gamma)=(w)$ and $\{ pq^*\> \vert \> tp=w=tq \}$ is a basis for $L_{\bf k}(\Gamma)$ by Proposition \ref{Path}(iv). If $\Gamma$ has a unique cycle $C$ and no sinks then $C$ has no exit. As above, each  $p_1q_1^* \in L_{\bf k}(\Gamma)$ with $tp_1=tq_1$ can be expressed as a linear combination of $pq^*$'s with $tp=sC=tq$ after applying (CK2) to $tp_1$ as needed. Hence $L_{\bf k}(\Gamma)=(sC)$ and $\{ pC^nq^*\> \vert \> n \in \mathbb{Z},  \> p,q \in P_C \}$ is a basis for $L_{\bf k}(\Gamma)$ by Proposition \ref{Path}(vi). This establishes the initial case of our induction.

\medskip
\medskip

For the general case, if $\Gamma$ has a sink then we pick a sink $w$,  otherwise we pick a cycle $C$ with no exit.
We consider the short exact sequence 
$$0 \longrightarrow (H) \longrightarrow L_{\bf k}(\Gamma) \longrightarrow L_{\bf k}(\Gamma/H) \longrightarrow 0$$
where $H$ is hereditary saturated closure of either $\{w\}$ or $\{sC \}$. By Proposition \ref{hereditary}(iii) either $(H)=(w)$ or $(H)=(sC)$. Note that $\Gamma/H$ contains all cycles and sinks in $\Gamma$ except for either $w$ or $C$ since no other sink or cycle can be in the hereditary saturated closure of $\{w\}$ or $\{sC\}$ by Fact \ref{H&S},
because $C$ has no exit. Since the total number of sinks and cycles in $\Gamma/H$ is less than those in $\Gamma$, by induction hypothesis $L_{\bf k}(\Gamma/H)$ has a ${\bf k}$-basis of the desired form. We can lift this basis to $L_{\bf k}(\Gamma)$ keeping the same form. The union of this with a basis of either $(w)=(H)$ or $(sC)=(H)$ used above gives the desired ${\bf k}$-basis for $L_{\bf k}(\Gamma)$.
\end{proof}

\medskip
\medskip

Now we use Theorem \ref{basis} to compute the Gelfand-Kirillov dimension of Leavitt path algebras.

\medskip
\medskip
Let $\mathbb{F}$ be a field, $A$ a finitely generated $\mathbb{F}$-algebra with 1. The Gelfand-Kirillov dimension of $A$ is:
 $$GK dim A =\limsup_{n \to \infty} \frac{log (dim(W^n))}{ log (n)}$$
 where $W$ is a finite dimensional $\mathbb{F}$-subspace generating $A$ with $1\in W$, and $W^n$ is the $\mathbb{F}$-span of  $n$-fold products of elements from $W$. The Gelfand-Kirillov dimension of $A$ is independent of the choice of $W$. The algebra $A$ has polynomial growth if and only if $GKdim(A)$ is finite.
 
\medskip
\medskip

The Leavitt path algebra $L_{\mathbb{F}}(\Gamma)$ is finitely generated if and only if $\Gamma$ is finite. In which case $L_{\mathbb{F}}(\Gamma)$ has $1=\sum_{v \in V} v$. If $\Gamma$ has intersecting cycles, say $C$ and $D$, with $sC=u=sD$ then the subalgebra generated by $C$ and $D$ is a free algebra in 2 noncommuting variables by Propositon \ref{Path}(i). Therefore $\mathbb{F}\langle C,D \rangle \>$ and hence $L_{\mathbb{F}}(\Gamma)$ have exponential growth, so $GKdim L_{\mathbb{F}}(\Gamma)$ is infinite. The converse is also true and 
we will prove a finer version which gives the Gelfand-Kirillov dimension of $L_{\mathbb{F}}(\Gamma)$ in terms of the digraph $\Gamma$, via a \textit{height} function defined on the poset of the sinks and cycles in a finite digraph $\Gamma$ whose cycles are pairwise disjoint. In particular, this also provides a new proof of \cite[Theorem 5]{aajz12}.

\medskip
\medskip
Let the cycles in $\Gamma$ be pairwise disjoint. Then the preorder $\leadsto$ defines a partial order on the set of sinks and cycles in $\Gamma$.  
We define the {\bf height} function on the sinks and the cycles in a finite digraph $\Gamma$: The \textit{height of a sink} is 0. The \textit{height of a cycle with no exit} is 1. 
The \textit{height of a cycle} $C$ \textit{with an exit} is recursively defined as: $\textit{ht} (C) = 2+max \{ \textit{ht}( x)  \> \vert \>  C \leadsto x\> , \> C\neq x \} $. (This also defines the height of the vertices of the complete reduction since they are identified with the sinks and the cycles of $\Gamma$.) We define the height of $\Gamma$ to be the maximum of the heights of its cycles or 0 if $\Gamma$ has no cycles. Since reduction preserves sinks and cycles, height is invariant under reduction.  

\medskip  
  \medskip
    
     \begin{theorem} \label{height}
If $\mathbb{F}$ is a field and $\Gamma$ is a finite digraph whose cycles are pairwise disjoint then $GK dim L_{\mathbb{F}}(\Gamma)=ht(\Gamma)$. 
\end{theorem}
\begin{proof}
Let $W$ be the span of all the cycles in $\Gamma$, their duals, all paths not containing any cycles and their duals. Note that $1=\sum v \in W$ and $W$ generates $L_{\mathbb{F}}(\Gamma)$ since every element in the basis $\mathcal{B}$ of Theorem \ref{basis} is contained in $W^n$ for some $n$.

\medskip
\medskip

Let $C_1 \leadsto C_2 \leadsto \cdots \leadsto C_k$  and $C_{k+m} \leadsto C_{k+m-1} \leadsto \cdots \leadsto C_{k+1} \leadsto C_k$ be distinct cycles with $p_1$ a path containing no cycles from $sC_1$ to $sC_2$, similarly $p_2$ from $sC_2$ to $sC_3$, $\cdots , p_{k-1}$ from $sC_{k-1}$ to $sC_k$, $p_k$ from $sC_{k+1}$ to $sC_k$, $\cdots$, $p_{k+m-1}$ from $sC_{k+m}$ to $sC_{k+m-1}$. The elements in $W^n$ of the form $$sC_1^{n_0}C_1^{n_1}p_1C_2^{n_2}p_2 \cdots C_k^{n_k}p_k^*(C_{k+1}^*)^{n_{k+1}} \cdots p_{k+m-1}^* (C_{k+m}^*)^{n_{k+m}}$$  $$=C_1^{n_1}p_1C_2^{n_2}p_2 \cdots C_k^{n_k}p_k^*(C_{k+1}^*)^{n_{k+1}} \cdots p_{k+m-1}^* (C_{k+m}^*)^{n_{k+m}}$$ where $n_0+n_1+\cdots +n_{k+m}=n-(k+m-1)$. These are distinct elements of the  basis 
 $\mathcal{B}$ and there are $\binom{n+1}{k+m}$ of them for $n$ large (this counting problem is equivalent to counting the number of ways of placing $n-(k+m-1)$ identical coins into $k+m$ distinct pockets). Hence $dim (W^n)$ has a lower bound which is a polynomial of degree $k+m$ in n. Thus  $GKdim L_{\mathbb{F}}(\Gamma) \geq k+m$.

\medskip
\medskip 
 If $ht(\Gamma)=2k-1$  then we can find distinct cycles $C_1 \leadsto C_2 \leadsto \cdots \leadsto C_k$. Setting  $C_{k+1}:=C_{k-1}$,  $C_{k+2}:=C_{k-2}$, $\cdots C_{2k-1} =C_1$, we get $GKdim L_{\mathbb{F}}(\Gamma) \geq 2k-1=ht(\Gamma)$.
                            
                             \medskip  
                             \medskip

 Similarly, given $C_1 \leadsto C_2 \leadsto \cdots \leadsto C_k \leadsto w$  and $C_{k+m} \leadsto C_{k+m-1} \leadsto \cdots \leadsto C_{k+1} \leadsto w $ be distinct cycles with $w$ a sink and $p_1$ a path containing no cycles from $sC_1$ to $sC_2$, also $p_2$ a path (containing no cycles) from $sC_2$ to $sC_3$, $\cdots , p_k$ from $sC_k$ to $w$, $p_{k+1}$ from $sC_{k+1}$ to $w$, $\cdots$, $p_{k+m}$ from $sC_{k+m}$ to $sC_{k+m-1}$.
The elements in $W^n$ of the form $$C_1^{n_1}p_1C_2^{n_2}p_2 \cdots C_k^{n_k}p_kw^{n_0}p_{k+1}^*(C_{k+1}^*)^{n_{k+1}} \cdots p_{k+m}^* (C_{k+m}^*)^{n_{k+m}}$$  $$=C_1^{n_1}p_1C_2^{n_2}p_2 \cdots C_k^{n_k}p_kp_{k+1}^*(C_{k+1}^*)^{n_{k+1}} \cdots p_{k+m}^* (C_{k+m}^*)^{n_{k+m}}$$ where $n_0+n_1+\cdots +n_{k+m}=n-(k+m)$.
 These are distinct elements of the  basis $\mathcal{B}$ and there are $\binom{n}{k+m}$ of them for $n$ large (this counting problem is equivalent to counting the number of ways of placing $n-(k+m)$ identical coins into $k+m+1$ distinct pockets). Hence $dim (W^n)$ has a lower bound which is a polynomial of degree $k+m$ in n. Thus  $GKdim L_{\mathbb{F}}(\Gamma) \geq k+m$. If $ht(\Gamma)=2k$  then we can find distinct cycles $C_1 \leadsto C_2 \leadsto \cdots \leadsto C_k \leadsto w$ where $w$ is a sink. Setting  $C_{k+1}:=C_k$,  $C_{k+2}:=C_{k-1}$, $\cdots C_{2k} =C_1$, we get $GKdim L_{\mathbb{F}}(\Gamma) \geq 2k=ht(\Gamma)$.               
 \medskip
 \medskip
 
To show that $ht(\Gamma) \geq GKdim L_{\mathbb{F}}(\Gamma) $, we note that an arbitrary element of $W^n$ for $n$ large is in the span of elements of the form: $$pq^*=p_0C_1^{n_1}p_1C_2^{n_2} \cdots p_{k-1}C_k^{n_k}p_{k}p_{k+1}^*(C_{k+1}^*)^{n_{k+1}} p_{k+2}^* \cdots (C_{k+m}^*)^{n_{k+m}}p_{k+m+1}^*$$ 
where $n_1+n_2+ \cdots + n_{k+m} \leq n$ and $p_0, p_1, \cdots p_{k+m+1}$ do not contain any cycles.

\medskip
\medskip

For the chains of the distinct cycles $C_1 \leadsto C_2 \leadsto \cdots \leadsto C_k$ and $ C_{k+m} \leadsto \cdots \leadsto C_{k+1}$ there are finitely many choices for each $p_i$ where $0\leq i \leq k+m+1$, the number of these choices depends on $k+m$, but not on $n$. Fixing $p_0, p_1, \cdots , p_{k+m+1}$, the number of possibilities for $n_1, n_2, \cdots n_{k+m}$ is bounded by a polynomial in $n$ of degree $k+m$ (counting the number of ways of placing at most $n$ identical coins into $k+m$ distinct pockets). Since $k+m \leq ht(\Gamma)$ we have a polynomial upper bound which is the sum of the polynomials of degree $\leq ht(\Gamma)$ corresponding to the choice of $p_0, p_1,\cdots, p_{k+m+1}$. (The bound on the number of these polynomials does not depend on $n$.) Therefore $GKdim L_{\mathbb{F}}(\Gamma) = ht(\Gamma)$.
\end{proof}

Some early results on Leavitt path algebras of finite digraphs are easy consequences of Theorem \ref{height}.  $ GKdim L_{\mathbb{F}}(\Gamma)=0=ht(\Gamma)$ if and only if $\Gamma$ is acyclic. Hence $L_{\mathbb{F}}(\Gamma)$ is finite dimensional if and only if $\Gamma$ is acyclic \cite{aam08}, since the Gelfand-Kirillov dimension of any algebra $A$ is 0 if and only if $ dim^{\mathbb{F}}(A)$ is finite.

\medskip
\medskip

$GKdim L_{\mathbb{F}}(\Gamma)=1 $ if and only if $ \Gamma$ has at least one cycle but cycles have no exits, that is, $ht(\Gamma)=1$. In this case $L_{\mathbb{F}}(\Gamma)$ is a direct sum of matrix algebras over $\mathbb{F} [x, x^{-1}] \> \>$ or  $\mathbb{F}$ \cite{aam08}. (This is the special case of Proposition \ref{aas} above when the coefficients are a field and $\Gamma$ is finite.) Hence representations of  $L_{\mathbb{F}}(\Gamma)$ with $GKdim L_{\mathbb{F}}(\Gamma)\leq 1$ are well understood.

\begin{examples} \label{ornek}
The graph $C^*$-algebras of the following digraphs (which are the completions of Leavitt path algebras with complex coefficients) are quantum disks, quantum spheres and quantum real projective spaces \cite{hs02}. The graph $C^*$-algebra of $qD^2$, the quantum 2-disk is also the Toeplitz algebra, the Leavitt path algebra of this digraph is isomorphic to the Jacobson algebra $\mathbb{F} \langle x,y \rangle / (1-xy)$.

\medskip
\medskip 
  $$ qD^{2n} : \qquad 
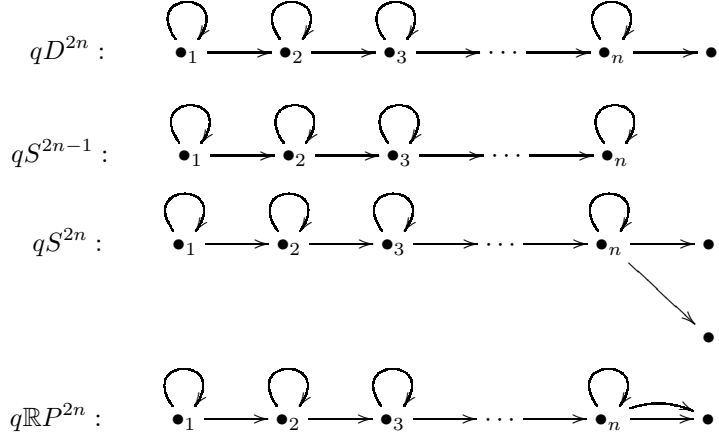
\xymatrix{ {\bullet}_1 \ar@(ul,ur)
 \ar@{->}[r]   &{\bullet}_2 \ar@(ul,ur)
 \ar@{->}[r]    & {\bullet}_3 \ar@(ul,ur)\ar@{->}[r] &\cdots \ar@{->}[r] & \bullet_n \ar@(ul,ur) \ar@{->}[r] 
 &\bullet }$$
 
$$qS^{2n-1} : \qquad \xymatrix{ {\bullet}_1 \ar@(ul,ur) \ar@{->}[r]   &{\bullet}_2 \ar@(ul,ur)  \ar@{->}[r]    & {\bullet}_3 \ar@(ul,ur)\ar@{->}[r] &\cdots \ar@{->}[r] & \bullet_n \ar@(ul,ur) }\qquad \qquad $$
 
 $$qS^{2n} : \qquad \xymatrix{ {\bullet}_1 \ar@(ul,ur)
 \ar@{->}[r]   &{\bullet}_2 \ar@(ul,ur)
 \ar@{->}[r]    & {\bullet}_3 \ar@(ul,ur)\ar@{->}[r] &\cdots \ar@{->}[r] & \bullet_n \ar@(ul,ur) \ar@{->}[r] \ar@{->}[dr] &\bullet \\
               &&& & & \bullet }$$
               
                  $$q\mathbb{R}P^{2n} : \qquad \xymatrix{ {\bullet}_1 \ar@(ul,ur)
 \ar@{->}[r]   &{\bullet}_2 \ar@(ul,ur)
 \ar@{->}[r]    & {\bullet}_3 \ar@(ul,ur)\ar@{->}[r] &\cdots \ar@{->}[r] & \bullet_n \ar@(ul,ur) \ar@{->}[r]  \ar@/^/[r]  &\bullet }\> \> \>\> $$

\medskip
\medskip
\noindent
 Note that all these digraphs are completely reduced and $ GKdim L_{\mathbb{F}}(\Gamma)=ht(\Gamma)$ is the dimension of the quantum space for these digraphs.
             \end{examples}

\begin{theorem} \label{k-height}
If $\Gamma$ is a finite digraph, $\mathbb{F}$ is a field and ${\bf k}$ is a commutative $\mathbb{F}$-algebra with 1 then $L_{{\bf k}}(\Gamma)$ is also an $\mathbb{F}$-algebra and

$$ GKdim L_{{\bf k}}(\Gamma)= GKdim L_{\mathbb{F}}(\Gamma) + GKdim ({\bf k}). $$
\end{theorem}

\begin{proof}
If $GKdim L_{\mathbb{F}}(\Gamma)$ or $GKdim ({\bf k})$ is infinite then $GKdim L_{\bf k}(\Gamma)$ is also infinite since the former are subalgebras of the latter. When both $GKdim L_{\mathbb{F}}(\Gamma)$ and $GKdim ({\bf k})$ are  finite, let $U$ be a finite dimensional generating $\mathbb{F}$-subspace of ${\bf k}$ containing 1 and let $W$ be the span of all the cycles in $\Gamma$, their duals, all paths not containing any cycles and their duals as in the proof of Theorem \ref{height}. We will use the generating subspace $U\otimes W$ of ${\bf k} \otimes_{\mathbb{F}} L_{\mathbb{F}}(\Gamma) \cong L_{{\bf k}}(\Gamma)$ which contains $1\otimes 1=1 $, to compute $GKdim L_{{\bf k}}(\Gamma)$. Now $dim^{\mathbb{F}} (U\otimes W)^n= (dim^{\mathbb{F}} U^n)( dim^{\mathbb{F}} W^n)$. Also $\limsup_{n \to \infty} \frac{log (dim(U^n))}{ log (n)}= GKdim ({\bf k})$ and, as shown in the proof of Theorem \ref{height}, $\lim_{n \to \infty} \frac{log (dim(W^n))}{ log (n)}=ht(\Gamma)$ yielding that $ GKdim L_{{\bf k}}(\Gamma)= ht(\Gamma) + GKdim ({\bf k}) =GKdim L_{\mathbb{F}}(\Gamma)+ GKdim ({\bf k})$.
\end{proof}

\medskip

\begin{remark} If $\mathbb{F}$ is a field, ${\bf k}$ is a commutative $\mathbb{F}$-algebra with 1 and $\Gamma$ is a finite digraph whose cycles are pairwise disjoint then
 $GKdim \> {\bf k}\Gamma \>= \lceil ht(\Gamma)/2 \rceil+ GKdim ({\bf k})$. The proof is similar to those of Theorem \ref{height} and Theorem \ref{k-height} but considerably easier since there are no dual paths and the set of paths is already a basis for the path algebra.
\end{remark}

If $A$ is a finitely generated algebra over a commutative domain {\bf k} then the Gelfand–Kirillov dimension of $A$ with respect to {\bf k} is defined to be

$$GKdim_{\bf k}(A) = \limsup_{n \to \infty} \frac{log \big(rank^{\bf k} (W^n)\big)}{log(n)}$$
where $W$ is a finitely generated ${\bf k}$-submodule of $A$ generating $A$ with $1\in W$, and $W^n$ is the ${\bf k}$-span of  $n$-fold products of elements from $W$ \cite[(E1.2.1)]{Bell}.

\begin{theorem}

If ${\bf k}$ is an integral domain and $\Gamma$ is a finite digraph whose cycles are pairwise disjoint then $GKdim_{\bf k}L_{\bf k}(\Gamma) =ht(\Gamma)$.
\end{theorem}
\begin{proof}
$GKdim_{\bf k}L_{\bf k}(\Gamma) =GKdim (\mathbb{F}\otimes A)= ht(\Gamma)$ where 
$\mathbb{F}$ is the field of fractions of {\bf k}. The first equality follows from the definition of $GKdim_{\bf k}$ since $rank^{\bf k} (W^n)= dim^{\mathbb{F}}(\mathbb{F}\otimes W)^n$ as in \cite[Lemma 3.1(i)]{Bell}. The second equality is from Theorem \ref{height}.
\end{proof}

\medskip
\medskip

\noindent
{\bf Acknowledgement}\\
Both authors were partially supported by TUBITAK grant 122F414.

\noindent
Ayten KO\c{C}\\
Department of Mathematics\\
Gebze Technical University\\
Gebze, T\"{U}RKİYE\\
E-mail: aytenkoc@gtu.edu.tr

\medskip
\medskip
\noindent
Murad \"{O}ZAYDIN\\
Department of Mathematics\\
University of Oklahoma\\
Norman, OK, USA\\
E-mail: mozaydin@ou.edu

\end{document}